\documentclass[12pt,reqno]{amsart}
\usepackage{amsmath}
\usepackage{amsthm}
\usepackage{amsfonts}
\usepackage{amssymb}
\usepackage{amsxtra}
\usepackage{mathrsfs}
\usepackage[hypertex]{hyperref}
\newtheorem{theorem}{Theorem}
\newtheorem*{theorem*}{Theorem}
\newtheorem{prop}[theorem]{Proposition}
\newtheorem{corollary}[theorem]{Corollary}
\theoremstyle{definition}
\newtheorem{definition}{Definition}
\newtheorem{example}{Example}

\newtheorem*{remark*}{Remark}
\DeclareMathOperator{\Ext}{Ext}
\DeclareMathOperator{\Tor}{Tor}
\DeclareMathOperator{\VdB}{V\!dB}
\DeclareMathOperator{\rk}{rk}

\newcommand*{\Ptens}{\mathop{\widehat\otimes}}
\newcommand*{\ptens}[1]{\mathop{\widehat\otimes}_{#1}}
\newcommand*{\tens}[1]{\mathop{\otimes}_{#1}}
\newcommand*{\Tens}{\mathop{\otimes}}

\newcommand*{\lar}{\leftarrow}

\newcommand*{\xla}{\xleftarrow}
\newcommand*{\xra}{\xrightarrow}
\newcommand*{\h}{\mathbf h}
\newcommand*{\lmod}{\mbox{-}\!\mathop{\mathsf{mod}}}
\newcommand*{\rmod}{\mathop{\mathsf{mod}}\!\mbox{-}}
\newcommand*{\bimod}{\mbox{-}\!\mathop{\mathsf{mod}}\!\mbox{-}}

\renewcommand*{\dh}{\mathop{\mathrm{dh}}}
\newcommand*{\db}{\mathop{\mathrm{db}}}
\newcommand*{\dg}{\mathop{\mathrm{dg}}}
\newcommand*{\wdh}{\mathop{\mathrm{w.dh}}}
\newcommand*{\wdb}{\mathop{\mathrm{w.db}}}
\newcommand*{\wdg}{\mathop{\mathrm{w.dg}}}
\newcommand*{\CC}{\mathbb C}
\newcommand*{\R}{\mathbb R}
\newcommand*{\N}{\mathbb N}
\newcommand*{\Z}{\mathbb Z}
\newcommand*{\TT}{\mathbb T}
\newcommand*{\bq}{\mathbf q}
\newcommand*{\cO}{\mathscr O}
\newcommand*{\cA}{\mathscr A}
\newcommand*{\cH}{\mathscr H}
\newcommand*{\cL}{\mathscr L}
\newcommand*{\cE}{\mathscr E}
\newcommand*{\eps}{\varepsilon}

\newcommand*{\reg}{\mathrm{reg}}
\newcommand*{\hol}{\mathrm{hol}}

\begin{document}
\title[Homological dimensions of quantum tori]{Homological dimensions
of smooth and\\ complex analytic quantum tori}
\author{A. Yu. Pirkovskii}
\address{Department of Nonlinear Analysis and Optimization\\
Faculty of Science\\
Peoples' Friendship University of Russia\\
Mikluho-Maklaya 6\\
117198 Moscow\\
Russia}
\email{pirkosha@sci.pfu.edu.ru, pirkosha@online.ru}
\thanks{Partially supported by the RFBR grant 08-01-00867.}
\date{}

\begin{abstract}
We survey some results on homological dimensions
of the algebraic, complex analytic, and smooth quantum tori.
Our main theorem states, in particular, that
the smooth and the complex analytic quantum $n$-tori
have global dimension $n$.
This contrasts with the result of McConnell and Pettit (1988) who proved that,
in the generic case, the algebraic quantum $n$-torus has global dimension $1$.
In this connection we also formulate some general theorems
on homological dimensions of nuclear Fr\'echet algebras.
\end{abstract}

\maketitle

\section{Introduction}
By a quantum torus one usually means an associative algebra which is, in a sense,
a ``noncommutative deformation'' of a function algebra on the $n$-torus.
The simplest example of a quantum torus is the algebra
$\cA_q$ generated by two invertibles $x,y$ subject to the relation
\begin{equation}
\label{q}
xy=qyx
\end{equation}
where $q$ is a nonzero scalar.
If $q=1$ then $\cA_q$ is just the algebra of Laurent polynomials in two variables
or, equivalently, the algebra of regular functions on the
algebraic $2$-torus $(\CC^\times)^2$, where $\CC^\times=\CC\setminus\{ 0\}$.

Relations like \eqref{q} naturally arise in quantum mechanics and go
back to H.~Weyl \cite{Weyl}; sometimes they are referred to as
``the canonical commutation relations in Weyl's form''.
The study of algebraic properties of $\cA_q$ was apparently initiated
by Wedderburn \cite{Wedd}.
Let us note that $\cA_q$ can be obtained from the Laurent polynomial algebra
$\cA_1$ via deformation quantization \cite{Rief_tori}.
From this perspective, $\cA_1$ becomes the ``classical limit''
of $\cA_q$ as $q\to 1$.

Similarly one defines other quantum tori, which are noncommutative
analogues of the algebras of regular, holomorphic, smooth, continuous,
and $L^\infty$-functions on the $n$-torus. Quantum tori play an important r\^ole in
noncommutative geometry \cite{Connes_book,Manin_topics} and
in the quantum group theory \cite{Br_Good,Lev_Soib}. They also naturally appear
in problems of quantum physics (quantum Hall effect \cite{Connes_book}, matrix models
in string theory \cite{CDS}, etc.); see also \cite{Rief-Schw,Weaver} and references therein.

Thus there are at least five natural versions of quantum tori:

\begin{itemize}
\item the {\em algebraic quantum torus}, which is a noncommutative analogue
of the algebra of Laurent polynomials in $n$ variables;
\item the {\em complex analytic quantum torus}, which is a noncommutative analogue
of the algebra of holomorphic functions on the complex algebraic
$n$-torus $(\CC^\times)^n$;
\item the {\em smooth quantum torus}, which is a noncommutative analogue
of the algebra of smooth functions on the real $n$-torus $\TT^n$;
\item the {\em topological quantum torus}, which is a noncommutative analogue
of the algebra of continuous functions on $\TT^n$;
\item the {\em measurable quantum torus}, which is a noncommutative analogue
of the algebra of $L^\infty$-functions on $\TT^n$.
\end{itemize}

Our goal is to present some results on homological dimensions
of the algebraic, complex analytic, and smooth quantum tori. The
results on the algebraic quantum torus are mostly due to
McConnell and Pettit \cite{McConnPett} and Brookes \cite{Brookes},
while the results on the complex analytic and smooth quantum tori are
due to the author. Before formulating the results, let us give the
definitions of the above-mentioned quantum tori.

\section{Preliminaries}
We will work over the field of complex numbers $\CC$.
All algebras are assumed to be associative and unital.

\subsection{The algebraic quantum torus}
Fix a complex $n\times n$-matrix $\bq=(q_{ij})_{1\le i,j\le n}$
such that $q_{ij}=q_{ji}^{-1}$ for all $i,j=1,\ldots ,n$.

\begin{definition}
{\em The algebraic quantum $n$-torus} is the algebra $\cO_\bq^\reg((\CC^\times)^n)$
with generators $x_1^{\pm 1},\ldots ,x_n^{\pm 1}$ and relations
\[
x_i x_i^{-1}=x_i^{-1}x_i=1,\quad x_i x_j=q_{ij} x_j x_i \quad (i,j=1,\ldots ,n).
\]
\end{definition}

In the commutative case (i.e., in the case where $q_{ij}=1$
for all $i,j$), $\cO_\bq^\reg((\CC^\times)^n)$ is just the algebra of Laurent polynomials
in $n$ variables, or, equivalently, the algebra
$\cO^\reg((\CC^\times)^n)$ of regular (in the sense of algebraic geometry)
functions on the algebraic $n$-torus $(\CC^\times)^n$.
In the general case, although $\cO_\bq^\reg((\CC^\times)^n)$ is clearly noncommutative,
one can easily show that the monomials
$x^\alpha=x_1^{\alpha_1}\cdots x_n^{\alpha_n}$ (where $\alpha_i\in\Z$)
form a basis of $\cO_\bq^\reg((\CC^\times)^n)$, so the underlying vector space
of $\cO_\bq^\reg((\CC^\times)^n)$ is still the space of Laurent polynomials.
Thus $\cO_\bq^\reg((\CC^\times)^n)$ can be viewed as the Laurent polynomial algebra
with a deformed multiplication. As we said above,
the study of the algebraic quantum torus was initiated by Wedderburn \cite{Wedd}
in the case $n=2$; for the general case, see \cite{McConnPett,Manin_topics}.

\subsection{The complex analytic quantum torus}
Let $\cO^\hol((\CC^\times)^n)$ denote the space of holomorphic functions
on $(\CC^\times)^n$ endowed with the topology of compact convergence.
Clearly, $\cO_\bq^\reg((\CC^\times)^n)$ is a dense subspace of
$\cO^\hol((\CC^\times)^n)$. It is natural to ask whether we can ``deform''
the usual pointwise multiplication on $\cO^\hol((\CC^\times)^n)$
in such a way that $\cO_\bq^\reg((\CC^\times)^n)$ become a subalgebra
of $\cO^\hol((\CC^\times)^n)$.
It is easy to see that the answer is positive provided that
$|q_{ij}|=1$ for all $i,j$. Indeed, identifying each function
$f\in \cO^\hol((\CC^\times)^n)$ with its Laurent expansion at $0$,
we get an isomorphism of topological vector spaces
\[
\cO^\hol((\CC^\times)^n)\cong\Bigl\{ a=\sum_{\alpha\in\Z^n} c_\alpha x^\alpha :
\| a\|_\rho=\sum_{\alpha\in\Z^n} |c_\alpha| \rho^{|\alpha|}<\infty\;\forall\rho>0\Bigr\}.
\]
Thus the standard topology on $\cO^\hol((\CC^\times)^n)$ is identical to the topology
determined by
the seminorms $\{\|\cdot\|_\rho : \rho>0\}$.
Now an easy computation shows that if $|q_{ij}|=1$, then the multiplication
on $\cO_\bq^\reg((\CC^\times)^n)$ is continuous with respect to the above family of seminorms,
and hence it uniquely extends by continuity to $\cO^\hol((\CC^\times)^n)$.
As a result, we get a new multiplication on
$\cO^\hol((\CC^\times)^n)$ making it into a topological algebra.

\begin{definition}[\cite{Pir_qfree}]
The algebra $\cO^\hol((\CC^\times)^n)$ endowed with the above multiplication
is called the {\em complex analytic quantum $n$-torus} and is denoted by
$\cO_\bq^\hol((\CC^\times)^n)$.
\end{definition}

\begin{remark*}
In \cite{Pir_qfree} we have shown that
$\cO_\bq^\hol((\CC^\times)^n)$ is the {\em Arens--Michael envelope}
of $\cO_\bq^\reg((\CC^\times)^n)$, i.e., the completion of
$\cO_\bq^\reg((\CC^\times)^n)$ with respect to the family of all submultiplicative seminorms.
Note that if $|q_{ij}|\ne 1$ for some $i,j$, then the Arens--Michael envelope of
$\cO_\bq^\reg((\CC^\times)^n)$ is zero (loc. cit.).
\end{remark*}

\subsection{The smooth quantum torus}
Consider the space $C^\infty(\TT^n)$ of smooth functions on the real $n$-torus $\TT^n$.
Recall that the standard topology on $C^\infty(\TT^n)$ is the topology of uniform
convergence of all derivatives.
The restriction map
\[
\cO^\hol((\CC^\times)^n)\to C^\infty(\TT^n),\quad f\mapsto f|_{\TT^n},
\]
is known to be injective and to have dense range.
Therefore $\cO_\bq^\hol((\CC^\times)^n)$ becomes a dense subspace of
$C^\infty(\TT^n)$. As above, it is easily seen that the usual pointwise
multiplication on $C^\infty(\TT^n)$ can be ``deformed'' in such a way that
$\cO_\bq^\hol((\CC^\times)^n)$ become a subalgebra of $C^\infty(\TT^n)$.
Indeed, identifying each function $f\in C^\infty(\TT^n)$
with its Fourier expansion, we get an isomorphism of topological vector spaces
\[
C^\infty(\TT^n)\cong
\Bigl\{ a=\sum_{\alpha\in\Z^n} c_\alpha x^\alpha :
\| a\|_k=\sum_{\alpha\in\Z^n} |c_\alpha| |\alpha|^k<\infty\;\forall k\in\Z_+\Bigr\}.
\]
Thus the standard topology on $\cO^\hol((\CC^\times)^n)$ is identical to the topology
determined by the seminorms $\{\|\cdot\|_k : k\in\Z_+\}$.
Now an easy computation shows that the multiplication
on $\cO_\bq^\hol((\CC^\times)^n)$ is continuous with respect to the above family of seminorms,
and hence it uniquely extends by continuity to $C^\infty(\TT^n)$.
As a result, we get a new multiplication on
$C^\infty(\TT^n)$ making it into a topological algebra.

\begin{definition}[M. Rieffel, \cite{Rief_tori_1}]
The algebra $C^\infty(\TT^n)$ endowed with the above multiplication is called
the {\em smooth quantum $n$-torus} and is denoted by $C^\infty_\bq(\TT^n)$.
\end{definition}

\subsection{The topological quantum torus}

\begin{definition}[G. Elliott, \cite{Ell_tori}]
The {\em topological quantum $n$-torus} is the universal $C^*$-algebra $C_\bq(\TT^n)$
generated by $n$ unitaries $u_1,\ldots ,u_n$ subject to the relations
$u_i u_j=q_{ij} u_j u_i\; (i,j=1,\ldots ,n)$.
\end{definition}

If $q_{ij}=1$ for all $i,j$, then $C_\bq(\TT^n)$ is isometrically $*$-isomorphic
to the algebra $C(\TT^n)$ of continuous functions on $\TT^n$.
Note that if $n=2$, then $C_\bq(\TT^n)$ is the {\em rotation algebra}
introduced by M.~Rieffel \cite{Rief_irr}.

\subsection{The measurable quantum torus}
Let $(\theta_{kl})_{1\le k,l\le n}$ be a real skew-symmetric matrix such that
$q_{kl}=\exp(2\pi i\theta_{kl})$ for all $k,l$. In what follows we identify
$\TT$ with $\R/\Z$ in the standard way. For each $k=1,\ldots ,n$, define
a unitary operator $U_k$ on $L^2(\TT^n)$ by
\[
(U_k f)(x_1,\ldots ,x_n)=\exp(2\pi i x_k) f\Bigl(x_1+\frac{\theta_{k1}}{2},\ldots ,
x_n+\frac{\theta_{kn}}{2}\Bigr) \qquad (f\in L^2(\TT^n)).
\]
An easy computation shows that $U_k U_l=q_{kl} U_l U_k$ for all $k,l$.
Therefore there exists a unique $*$-representation $\pi$ of $C_\bq(\TT^n)$ on
$L^2(\TT^n)$ such that $\pi(u_k)=U_k\; (k=1,\ldots ,n)$.

\begin{definition}[N. Weaver, \cite{Weaver}]
The weak operator closure of $\mathop{\mathrm{Im}}\pi$ is called
the {\em measurable quantum torus} and is denoted by $L^\infty_\bq(\TT^n)$.
\end{definition}

It is clear from the above definition that if $q_{kl}=1$ for all $k,l$, then
$L^\infty_\bq(\TT^n)$ is isomorphic to $L^\infty(\TT^n)$.

In summary, for every complex $n\times n$-matrix $\bq=(q_{ij})$
satisfying $q_{ij}=q_{ji}^{-1}$ and $|q_{ij}|=1$, we have a chain of algebras
\[
\cO_\bq^\reg((\CC^\times)^n)\subset\cO_\bq^\hol((\CC^\times)^n)\subset
C^\infty_\bq(\TT^n)\subset C_\bq(\TT^n)\subset L^\infty_\bq(\TT^n).
\]
Below we will concentrate mostly on the complex analytic quantum
torus $\cO_\bq^\hol((\CC^\times)^n)$
and on the smooth quantum torus $C^\infty_\bq(\TT^n)$.
In this connection, we will also recall some related results
on the algebraic quantum torus
$\cO_\bq^\reg((\CC^\times)^n)$, obtained by McConnell and Pettit \cite{McConnPett}
and Brookes \cite{Brookes}.

Unfortunately, the results we are going to present do not extend to
the topological quantum torus $C_\bq(\TT^n)$
and to the measurable quantum torus $L^\infty_\bq(\TT^n)$. The main difficulty
in studying homological properties of $A=C_\bq(\TT^n)$
and $A=L^\infty_\bq(\TT^n)$ is that the
completed projective tensor product $A\Ptens A$
is a rather complicated Banach space.
Take, for instance, the simplest situation $n=1$, in which case
$C_\bq(\TT^n)$ is just the algebra $C(\TT)$ of continuous functions on the circle.
It is known that $C(\TT)\Ptens C(\TT)$ is a proper subspace of $C(\TT^2)$,
and that the projective tensor norm on $C(\TT)\Ptens C(\TT)$ is strictly
stronger than the uniform norm inherited from $C(\TT^2)$.
Moreover, given a function $f\in C(\TT^2)$, there is no effective way
to determine whether or not $f$ belongs to $C(\TT)\Ptens C(\TT)$,
and even if it does, then there is no effective way to compute its
projective tensor norm. A similar problem occurs with $L^\infty(\TT)$.
In contrast, spaces of smooth functions
behave well under the projective tensor product in the sense that,
given smooth manifolds $M$ and $N$, there is a topological isomorphism
$C^\infty(M)\Ptens C^\infty(N)\cong C^\infty(M\times N)$.
A similar property holds for spaces of holomorphic functions on complex
manifolds (see \cite{Groth} for details).

\section{Homological dimensions}
Homological dimensions of associative algebras can be defined in
at least two different settings. The first one is the classical
homological algebra of Cartan and Eilenberg \cite{CE}, i.e.,
homological algebra in categories of modules over rings.
The second one is a version of homological algebra in categories of functional
analysis, specifically in categories of locally convex topological modules
over locally convex topological algebras. This theory, also known as
{\em topological homology}, was developed in the early 1970ies by
Helemskii (see, e.g., \cite{X_dg}) in the special case of Banach algebras.
A few years later a similar theory was independently discovered
by Kiehl and Verdier \cite{KV} and by Taylor \cite{T1}
in the context of more general topological algebras.
Let us briefly recall the basics of this theory.
For details, we refer to Helemskii's monograph \cite{X1}.

To be definite, we will work only with Fr\'echet modules over Fr\'echet algebras.
Recall that a {\em Fr\'echet algebra} is an algebra $A$ endowed with a topology
making $A$ into a Fr\'echet space (i.e., a complete, metrizable locally convex space)
in such a way that the multiplication $A\times A\to A$ is continuous.
A {\em left Fr\'echet $A$-module} is a left $A$-module $X$ endowed with
a Fr\'echet space topology in such a way that the action $A\times X\to X$
is continuous.
Left Fr\'echet $A$-modules and their continuous morphisms form a category
denoted by $A\lmod$. Given $X,Y\in A\lmod$,
the space of morphisms from $X$ to $Y$ will be denoted by $\h_A(X,Y)$.
The categories $\rmod A$ and $A\bimod A$ of right
Fr\'echet $A$-modules and of Fr\'echet $A$-bimodules are defined similarly.

The basic constructions of topological homology
mostly parallel their classical counterparts from \cite{CE}.
However, there is a crucial difference stemming from the fact that
the categories of Fr\'echet modules are not abelian.
The difference is that, instead of considering arbitrary exact sequences
of $A$-modules, one should restrict to those sequences which are
``admissible'' in the following sense.
An exact sequence of Fr\'echet modules is {\em admissible} if
it splits in the category of topological vector spaces, i.e., if it has
a contracting homotopy consisting of continuous linear maps.
By using admissible sequences instead of arbitrary exact sequences,
one can adapt most basic notions of the classical homological algebra
to the context of Fr\'echet modules.
For example, a left Fr\'echet $A$-module $P$ is {\em projective} if the functor
$\h_A(P,-)$ is exact in the sense that it takes admissible sequences
of Fr\'echet $A$-modules to exact sequences of vector spaces.
A left Fr\'echet $A$-module $F$ is {\em flat} if the projective tensor product functor
$(-)\ptens{A} F$ (see \cite{X1}) is exact in the same sense as above.
It is known that every projective Fr\'echet module is flat.

A {\em resolution} of $X\in A\lmod$ is a pair $(P_\bullet,\eps)$
consisting of a nonnegative chain complex
$P_\bullet$ in $A\lmod$ and a morphism $\eps\colon P_0\to X$ making the sequence
$P_\bullet\xra{\eps} X\to 0$ into an admissible complex.
If all the $P_i$'s are projective (respectively, flat), then
$(P_\bullet,\eps)$ is called a {\em projective resolution}
(respectively, a {\em flat resolution}) of $X$.
It is a standard fact that $A\lmod$ has {\em enough projectives},
i.e., each left Fr\'echet $A$-module has a projective resolution.
The same is true of $\rmod A$ and $A\bimod A$.

By using the above fact, we may define derived functors on $A\lmod$,
in particular, the functors $\Ext$ and $\Tor$.
Let $X$ be a left Fr\'echet $A$-module, and let $P_\bullet\to X$ be a projective
resolution of $X$. Given $Y\in A\lmod$, the $n$th cohomology of the cochain complex
$\h_A(P_\bullet,X)$ is denoted by $\Ext^n_A(X,Y)$. Similarly, if
$Y\in\rmod A$, then the $n$th homology of the chain complex
$Y\ptens{A} P_\bullet$ is denoted by $\Tor_n^A(Y,X)$.
The spaces $\Ext$ and $\Tor$ do not depend on the choice of the projective
resolution $P_\bullet$ because all projective resolutions of $X$
are homotopy equivalent.

An important special case of $\Tor$ and $\Ext$ is Hochschild homology
and cohomology. Given a Fr\'echet $A$-bimodule $X$, the space
$\Ext^n_{A-A}(A,X)$ (here ``$A-A$'' means that we are dealing with the
$\Ext$ functor on $A\bimod A$) is called the {\em $n$th Hochschild cohomology}
of $A$ with coefficients in $X$ and is denoted by $\cH^n(A,X)$.
Similarly, the {\em $n$th Hochschild homology}
of $A$ with coefficients in $X$ is the space
$\cH_n(A,X)=\Tor_n^{A-A}(X,A)$.

Let $X\in A\lmod$. The {\em projective homological dimension} of $X$,
denoted by $\dh_A X$,
is the least integer $n\in\Z_+$ such that $X$ has a projective resolution of the form
\[
0 \lar X\lar P_0\lar P_1\lar \cdots \lar P_n\lar 0.
\]
If there is no such $n$, one sets $\dh_A X=\infty$.
If we replace the words ``projective resolution'' by ``flat resolution'',
then we get the definition of the
{\em weak homological dimension} of $X$, denoted $\wdh_A X$.
Clearly, we have $\dh_A X=0$ (respectively, $\wdh_A X=0$)
if and only if $X$ is projective (respectively, flat).
Since every projective module is flat, we clearly have $\wdh_A X\le\dh_A X$.
The projective dimension of $X$ can also be defined as the least integer
$n\in\Z_+$ such that $\Ext^{n+1}_A(X,Y)=0$
for all $Y\in A\lmod$. Similarly, the weak dimension of $X$ is the
least integer $n\in\Z_+$ such that $\Tor_{n+1}^A(Y,X)=0$
and $\Tor_n^A(Y,X)$ is Hausdorff for all $Y\in\rmod A$.

Given a Fr\'echet algebra $A$, the {\em global dimension} and
the {\em weak global dimension} of $A$
are defined by
\begin{align*}
\dg A&=\sup\{ \dh\nolimits_A X \,|\, X\in A\lmod\},\\
\wdg A&=\sup\{ \wdh\nolimits_A X \,|\, X\in A\lmod\}.
\end{align*}
The {\em bidimension} and the {\em weak bidimension} of
$A$ are defined by
\begin{align}
\label{db_H}
\db A&=\dh\nolimits_{A-A} A
=\min\{ n\in\Z_+ \mid \cH^{n+1}(A,M)=0\quad\forall\, M\in A\bimod A\},\\
\notag
\wdb A&=\wdh\nolimits_{A-A} A
=\min\left\{
n\in\Z_+ \,\left| \;
\parbox{42mm}{$\cH_{n+1}(A,M)=0$ and\\ $\cH_n(A,M)$
is Hausdorff}\;\forall\, M\in A\bimod A\right.
\right\}.
\end{align}
We clearly have $\wdg A\le\dg A$ and $\wdb A\le\db A$.
It is also true (but less obvious) that $\dg A\le\db A$ and $\wdg A\le\wdb A$.

Apart from the functional-analytic version of homological algebra
that we have just described, we will also use its
purely algebraic prototype, i.e.,
the Cartan--Eilenberg homological algebra in categories of modules over
algebras not endowed with any topology. Recall that, in order to define
homological dimensions in the purely algebraic setting,
we should repeat the above definitions with admissible sequences replaced by exact sequences
and the completed projective tensor product,
$\ptens{A}$, replaced by the algebraic tensor product, $\tens{A}$.
Also, the conditions that certain $\Tor$-spaces are required to be Hausdorff
(see the above definitions of $\wdh$ and $\wdb$) are now meaningless
and should be omitted. The reason why these conditions are essential
in the Fr\'echet algebra setting stems from the fact that
the functor $\ptens{A}$ is not right exact (in contrast to
the functor $\tens{A}$), and so $\ptens{A}$ is not isomorphic in general
to the derived functor $\Tor_0$.

In what follows, when dealing with homological dimensions of
the quantum tori, we will consider the complex analytic and the smooth quantum tori
as Fr\'echet algebras, while the algebraic quantum torus will be
considered as ``just an algebra''. Thus, for example, the symbol
``$\dg$'' will have different meanings when applied
to $\cO_\bq^\reg((\CC^\times)^n)$ and to $\cO_\bq^\hol((\CC^\times)^n)$.
We hope that this will not lead to confusion.

It turns out that the bidimensions of the quantum tori is much easier
to compute than their global dimensions.
The reason is that
the Hochschild homology and cohomology of the quantum tori satisfy
a relation resembling the classical Poincar\'e isomorphism
in the topology of manifolds.
This relation was first systematically studied by
M.~Van den Bergh \cite{VdB}, so we call it {\em Van den Bergh's condition}.

\section{Algebras satisfying Van den Bergh's condition}

Let $A$ be a Fr\'echet algebra. A bimodule $U\in A\bimod A$
is said to be \emph{invertible} if there exists a bimodule
$U^{-1}\in A\bimod A$ such that
\[
U\ptens{A} U^{-1}\cong U^{-1}\ptens{A} U\cong A
\]
as Fr\'echet $A$-bimodules.

Here is an example of an invertible bimodule. Let $\alpha$ be an automorphism
of $A$. Denote by $A_\alpha$ the Fr\'echet space $A$ with
an $A$-bimodule structure given by
\[
a\cdot b=ab,\quad b\cdot a=b\alpha(a)\quad (a\in A,\; b\in A_\alpha).
\]
It is easy to check that $A_\alpha$ is invertible and that
$A_\alpha^{-1}=A_{\alpha^{-1}}$.

\begin{definition}
We say that $A$ satisfies {\em Van den Bergh's condition
$\VdB(n)$} (where $n\in\N$) if there exists an invertible bimodule
$U\in A\bimod A$ such that
\begin{equation}
\label{VdB}
\cH^i(A,X)\cong \cH_{n-i}(A,U\ptens{A} X)\quad\text{for all } X\in A\bimod A.
\end{equation}
The bimodule $U$ will be called a {\em twisting bimodule}.
\end{definition}

Of course, a similar definition (with $\ptens{A}$ replaced by $\tens{A}$)
makes sense for algebras not endowed with any topology.
In this context, the above condition was introduced and studied by
M.~Van den Bergh \cite{VdB}.
In the setting of Fr\'echet algebras, Van den Bergh's condition
was first used presumably by the author \cite{Pir_Nova}.

\begin{prop}
\label{prop:VdB_db}
If $A$ satisfies $\VdB(n)$, then $\db A=n$.
\end{prop}
\begin{proof}
Since $\cH_i\equiv 0$ for all $i<0$,
condition \eqref{VdB} implies that $\cH^i(A,X)=0$ for all $i>n$ and all
$X\in A\bimod A$. This means exactly that $\db A\le n$ (see \eqref{db_H}).
On the other hand, it is known that for each Fr\'echet algebra $A$
and each $X\in A\lmod$, $Y\in \rmod A$ there is an isomorphism
$\cH_i(A,X\Ptens Y)\cong\Tor_i^A(Y,X)$. Therefore,
\[
\cH^n(A,U^{-1}\Ptens A)\cong \cH_0(A,A\Ptens A)\cong \Tor_0^A(A,A)\cong A\ne 0,
\]
which shows that $\db A=n$.
\end{proof}

Here are some examples of algebras satisfying $\VdB(n)$.

\begin{example}
The polynomial algebra $A=\CC[x_1,\ldots ,x_n]$ satisfies
$\VdB(n)$ (in the purely algebraic sense) with $U=A$.
This was first observed apparently by J.~L.~Taylor \cite{T2}
and easily follows from the fact that $A$
has a bimodule Koszul resolution.
\end{example}

\begin{example}
\label{example:Taylor}
The algebra $C^\infty(D)$ of smooth functions on an open subset
$D\subset\R^n$ and the algebra $\cO^\hol(D)$ of holomorphic functions
on a polydomain $D\subset\CC^n$ satisfy $\VdB(n)$
(as Fr\'echet algebras) with $U=A$.
This was proved by Taylor \cite{T2} by using bimodule Koszul resolutions,
like in the previous example. Of course, the main point here is that
the Koszul resolution is not only exact, but is also admissible.
Note that if $D\subset\CC^n$ is a domain of holomorphy, then the
bimodule Koszul resolution of $\cO(D)$ is still exact \cite{T2},
but the question of whether it is admissible seems to be open.
\end{example}

\begin{example}
The algebra $C^\infty(M)$ of smooth functions on a real manifold $M$
satisfies $\VdB(n)$ with $n=\dim(M)$ and $U=T^n(M)$, the module
of smooth $n$-polyvector fields on $M$ \cite{Pir_Nova}.
\end{example}

\begin{example}
A similar result holds for the algebra of regular functions and for
the algebra of holomorphic functions on a nonsingular affine algebraic
variety \cite{Pir_Mall}. It is tempting to conjecture that
the algebra of holomorphic functions on any Stein manifold $M$
satisfies Van den Bergh's condition, but this question is open
even in the case where $M$ is a domain of holomorphy in $\CC^n$.
\end{example}

For a number of other examples (in the purely algebraic context),
see Van den Bergh's paper \cite{VdB}. In particular, he shows that
condition $\VdB(n)$ holds for many Koszul algebras and for many ``almost commutative''
filtered algebras, such as, for example, the universal enveloping algebra
of a finite-dimensional Lie algebra.

We will see below that the algebras
$\cO_\bq^\reg((\CC^\times)^n)$, $\cO_\bq^\hol((\CC^\times)^n)$, and
$C^\infty_\bq(\TT^n)$ also satisfy $\VdB(n)$.
To this end, it is convenient to use ``quantized'' versions of
bimodule Koszul resolutions.

\section{Bimodule Koszul resolutions and the bidimensions of the quantum tori}

Let $A$ denote any of the algebras $\cO_\bq^\reg((\CC^\times)^n)$,
$\cO_\bq^\hol((\CC^\times)^n)$, or $C^\infty_\bq(\TT^n)$.
For each $p=0,\ldots ,n$, consider the $A$-bimodule
$K_p=A\Tens\bigwedge^p\CC^n\Tens A$, where $\Tens$ stands for the usual tensor product
(over $\CC$) in the case where $A=\cO_\bq^\reg((\CC^\times)^n)$
and for the completed projective tensor product in the case where
$A=\cO^\hol_\bq((\CC^\times)^n)$ or $A=C^\infty_\bq(\TT^n)$.
Fix a basis $e_1,\ldots ,e_n$ in $\CC^n$ and consider the chain complex
\begin{equation}
\label{Kosz_tor}
0 \lar A \xla{\mu_A} K_0 \xla{d} K_1 \xla{d} \cdots
\xla{d} K_n \lar 0,
\end{equation}
where $\mu_A\colon K_0=A\Tens A\to A$ is the multiplication on $A$, and the
differential $d\colon K_p\to K_{p-1}$ is given by
\begin{align*}
d(a\otimes e_{i_1}\ldots e_{i_p}\otimes b)
=\sum_{k=1}^p (-1)^{k-1}
&\biggl(\Bigl(\prod_{s<k} q_{i_s i_k}\Bigr) ax_{i_k}
\otimes e_{i_1}\ldots\hat e_{i_k}\ldots e_{i_p}\otimes b\\
&-\Bigl(\prod_{s>k} q_{i_k i_s}\Bigr) a\otimes e_{i_1}\ldots\hat e_{i_k}\ldots e_{i_p}
\otimes x_{i_k} b \biggr)
\end{align*}
for $a,b\in A$ and $1\le i_1<\ldots <i_p\le n$.

The following theorem is essentially due to R.~Nest \cite{Nest}
and L.~A.~Takhtajan \cite{Takht}. Although they considered only the case
where $A=C^\infty_\bq(\TT^n)$, their proofs remain valid for
$\cO_\bq^\reg((\CC^\times)^n)$ and $\cO_\bq^\hol((\CC^\times)^n)$ as well.
For $A=\cO_\bq^\reg((\CC^\times)^n)$, a similar result was obtained in
\cite{GG,Wambst_tor}.

\begin{theorem*}[R.~Nest \cite{Nest}, L.~A.~Takhtajan \cite{Takht}]
The complex \eqref{Kosz_tor} is exact.
Moreover, if $A$ is either $\cO_\bq^\hol((\CC^\times)^n)$
or $C^\infty_\bq(\TT^n)$, then \eqref{Kosz_tor} is admissible. Therefore
the complex $K_\bullet=(K_p,d_p)$ augmented by $\mu_A$ is a projective resolution
of $A$ in $A\bimod A$.
\end{theorem*}

The resulting resolution $(K_\bullet,\mu_A)$ is called the {\em bimodule Koszul
resolution} of $A$.

The next proposition is proved by a direct computation.

\begin{prop}
Let $\alpha$ be the automorphism of $A$ uniquely determined by
\[
\alpha(x_j)=\prod_{i>j} q_{ij} x_j \quad (j=1,\ldots ,n).
\]
Then for each $X\in A\bimod A$ there exists a chain isomorphism
\begin{equation}
\label{Kosz_sym}
\h_{A-A}(K_\bullet,X)\cong ({_\alpha} X\tens{A-A} K_\bullet)[-n].
\end{equation}
\end{prop}

The above bimodule ${_\alpha} X$ is defined in a similar fashion
to the bimodule $A_\alpha$ (see above). The symbol
$[-n]$, as usual, denotes the right shift by
$n$ degrees.

By taking the cohomology of \eqref{Kosz_sym} and by using the obvious isomorphism
${_\alpha} X\cong A_\alpha\tens{A} X$, we obtain the following.

\begin{corollary}
\label{cor:tori_VdB}
The algebras $\cO_\bq^\reg((\CC^\times)^n)$, $\cO_\bq^\hol((\CC^\times)^n)$, and
$C^\infty_\bq(\TT^n)$ satisfy $\VdB(n)$ with twisting bimodule $A_\alpha$.
\end{corollary}

Together with Proposition~\ref{prop:VdB_db}, this yields a bidimension formula
for the quantum tori.

\begin{corollary}
$\db\cO_\bq^\reg((\CC^\times)^n)
=\db\cO_\bq^\hol((\CC^\times)^n)=\db C^\infty_\bq(\TT^n)=n$.
\end{corollary}

\section{The global dimension of the algebraic quantum torus}

Computing the global dimensions of quantum tori is
considerably more difficult than computing their bidimensions.
In the case of the algebraic quantum torus, this problem was solved
by J.~C.~McConnell and J.~J.~Pettit \cite{McConnPett}. A more transparent
solution was subsequently given by C.~J.~B.~Brookes.

\begin{theorem*}[C. J. B. Brookes \cite{Brookes}]
For a subgroup $H\subset\Z^n$, let $A_H$ denote the subalgebra of
$\cO_\bq^\reg((\CC^\times)^n)$ spanned by the monomials
$x^\alpha=x_1^{\alpha_1}\cdots x_n^{\alpha_n}\; (\alpha\in H)$.
Then
\[
\dg\cO_\bq^\reg((\CC^\times)^n)=\max\{\,\rk H : A_H\text{ \upshape is commutative }\}.
\]
\end{theorem*}

It may happen that $A_H$ is commutative only in the extreme case where
$H$ is cyclic; in this case we have $\dg\cO_\bq((\CC^\times)^n)=1$.
As was shown by McConnell and Pettit, this case is in fact generic:

\begin{theorem*}[J.~C.~McConnell and J.~J.~Pettit \cite{McConnPett}]
Suppose that the multiplicative subgroup of $\CC^\times$ generated by
the $q_{ij}$'s has maximal possible rank, namely $\frac{q(q-1)}{2}$.
Then $\dg \cO_\bq^\reg((\CC^\times)^n)=1$.
\end{theorem*}

This theorem implies, in particular, that $\dg\cA_q=1$ if $q$ is not a root of unity.
See also \cite[7.11.3]{MR} and references therein.

To complete the homological picture of the algebraic quantum torus,
let us observe that $\wdg \cO_\bq^\reg((\CC^\times)^n)=\dg \cO_\bq^\reg((\CC^\times)^n)$
and $\wdb \cO_\bq^\reg((\CC^\times)^n)=\db \cO_\bq^\reg((\CC^\times)^n)$
because $\cO_\bq^\reg((\CC^\times)^n)$ is noetherian.

\section{The global dimensions of the complex analytic and smooth quantum tori}

At the first glance, the three above versions of quantum $n$-tori
(i.e., the algebraic, complex analytic, and smooth quantum $n$-tori)
look very similar to each other. Indeed, we have already seen that
all of them have bimodule Koszul resolutions, satisfy Van den Bergh's condition
$\VdB(n)$, and have bidimension $n$.
It is natural to expect that their global dimensions should also be equal to each
other. This is not the case, however. A crucial difference between
the algebraic and locally convex (i.e., complex analytic and smooth) quantum
tori is that the latter are \emph{nuclear} Fr\'echet spaces.

Let us recall that nuclear locally convex spaces were introduced by
A.~Grothendieck in the early 1950ies. We will not give the definition
of nuclear spaces here, referring the reader to standard books on
topological vector spaces (see, e.g., \cite{Sch,Pietsch}).
The class of nuclear spaces is rather
large and contains, in particular, the spaces of smooth and holomorphic functions
on real and complex manifolds, as well as many spaces of distributions.
On the other hand, a normed space is nuclear only if it is finite-dimensional.

One of the main advantages of nuclear spaces is that they often behave in
much the same way as finite-dimensional spaces.
For example, all closed bounded subsets of a complete nuclear space are
compact. Another example: if $X$ is a nuclear Fr\'echet space,
and $Y$ is any complete locally convex space, then
the space $\cL(X,Y)$ of continuous linear maps from $X$ to $Y$
is isomorphic to the projective tensor product $X^*\Ptens Y$, where
$X^*$ is the strong dual of $X$. In the setting of linear algebra,
a similar assertion holds only in the case where one of the spaces
$X,Y$ is finite-dimensional.

To formulate our next result, let us recall that an
{\em Arens--Michael algebra} is a complete topological algebra
$A$ such that the topology on $A$ can be determined by a family
of submultiplicative seminorms (i.e., seminorms $\|\cdot\|$
satisfying $\| ab\|\le\| a\|\| b\|$ for all $a,b\in A$).
Equivalently, an Arens--Michael algebra is an inverse limit
of Banach algebras. The latter assertion is often referred to as
{\em the Arens--Michael decomposition theorem}.

Most ``natural'' topological algebras (although not all of them)
are Arens--Michael algebras. Clearly, each Banach algebra is an
Arens--Michael algebra. The algebras of continuous functions on topological
spaces and the algebras of smooth and holomorphic functions on real and complex
manifolds are also Arens--Michael algebras.
On the other hand, the algebra $\cE'(\R^n)$ of compactly supported distributions
on $\R^n$ is not an Arens--Michael algebra.
For our purposes, an important fact is that
$\cO_\bq^\hol((\CC^\times)^n)$ and
$C^\infty_\bq(\TT^n)$ are nuclear Fr\'echet--Arens--Michael algebras.

\begin{theorem}
Let $A$ be a nuclear Fr\'echet--Arens--Michael algebra satisfying $\VdB(n)$. Then
\[
\dg A=\db A=\wdg A=\wdb A=n.
\]
\end{theorem}

Together with Corollary \ref{cor:tori_VdB}, this yields the following.

\begin{corollary}
\label{cor:dgdb_qtori}
Let $A$ be either $C^\infty_\bq(\TT^n)$ or $\cO_\bq^\hol((\CC^\times)^n)$. Then
\[
\dg A=\db A=\wdg A=\wdb A=n.
\]
\end{corollary}

It is interesting to compare the latter result with the above theorems of
McConnell--Pettit and Brookes. We see that,
while the global dimension of the algebraic quantum $n$-torus
can be any number between $1$ and $n$, the global dimensions of the
smooth and complex analytic quantum $n$-tori are always
equal to $n$.

\section{Global dimension versus bidimension}

In this final section we discuss some general results on homological dimensions
of nuclear Fr\'echet algebras. We have already noted above that
for each Fr\'echet algebra $A$ one has $\dg A\le\db A$ and $\wdg A\le\wdb A$.
It is natural to ask whether any of these inequalities can be strict.
This problem was explicitly formulated by A.~Ya.~Helemskii
\cite{Hel_probl_1} and is still open.
In all concrete cases where the above dimensions are known we actually have
$\dg A=\db A$ and $\wdg A=\wdb A$.
It is interesting to compare this phenomenon with the classical homological
algebra, where algebras with $\dg A<\db A$ or $\wdg A<\wdb A$
exist in abundance. For instance, the algebra
$A=\CC(t)$ of rational functions satisfies
$\dg A=\wdg A=0$, because $A$ is a field and all $A$-modules are projective.
On the other hand, it is known that $\wdg A=\wdb A=1$.
Note that a similar example cannot be constructed within the framework
of Fr\'echet algebras due to the Gelfand--Mazur--\.Zelazko theorem \cite{Zel_metr},
which states that every Fr\'echet division algebra is isomorphic to $\CC$.
Another purely algebraic example is the algebraic quantum torus
$A=\cO_\bq^\reg((\CC^\times)^n)$; we have already noted above that
$\db A=\wdb A=n$, while $\dg A=\wdg A$ can be any integer between
$1$ and $n$. Corollary~\ref{cor:dgdb_qtori} shows that this example
apparently has no analogue in the Fr\'echet algebra context.

In view of the above-mentioned problem, it seems natural
to establish the equality $\dg A=\db A$ or $\wdg A=\wdb A$
if not for all Fr\'echet algebras (which is rather doubtful),
at least for some natural and sufficiently large class of them.
One such class is given by the next theorem.

\begin{theorem}
\label{thm:wdg=wdb}
Let $A$ be a nuclear Fr\'echet--Arens--Michael algebra. Suppose that $\wdb A<\infty$.
Then $\wdg A=\wdb A$.
\end{theorem}

The proof is based on the above-mentioned Arens--Michael decomposition theorem,
on some results of V.~P.~Palamodov \cite{Pal_proj1}
on the vanishing of the derived inverse limit functor $\varprojlim^1$,
and on the author's results \cite{Pir_injdim}
on factorization of nuclear operators.

It is natural to ask whether Theorem~\ref{thm:wdg=wdb} can be extended
to the ``strong'' dimensions $\dg$ and $\db$.
Unfortunately, so far we have only an essentially weaker result
on $\dg$ and $\db$. Let us say that a Fr\'echet algebra
$A$ is {\em of finite type} if $A$ has a projective resolution in
$A\bimod A$ consisting of finitely generated bimodules.
For example, the algebra $C^\infty(D)$ of smooth functions on an open set
$D\subset\R^n$ and the algebra $\cO^\hol(D)$ of holomorphic functions
on a polydomain $D\subset\CC^n$ are of finite type,
because they have bimodule Koszul resolutions (see Example~\ref{example:Taylor}).
As was shown by A.~Connes \cite{NCDG}, the algebra
$C^\infty(M)$ of smooth functions on a compact manifold $M$ is of finite type
provided that $M$ has a nowhere vanishing vector field.
The algebra $\cO^\hol(V)$ of holomorphic functions on a nonsingular
affine algebraic variety $V$ is also of finite type \cite{Pir_Mall}.
Nest--Takhtajan's theorem (see above) implies that
$C^\infty_\bq(\TT^n)$ and $\cO_\bq^\hol((\CC^\times)^n)$ are of finite type.
For more examples, see \cite{Pir_qfree}.

\begin{theorem}
\label{thm:dg=db}
Let $A$ be a nuclear Fr\'echet--Arens--Michael algebra of finite type.
Suppose that $\db A<\infty$.
Then $\dg A\le\db A\le\dg A+1$.
\end{theorem}

Theorems \ref{thm:wdg=wdb} and \ref{thm:dg=db} may be compared with
the situation in the classical homological algebra, where similar results
seem to exist only for finite-dimensional algebras \cite{Eilenberg,Ausl,Happel}.
Thus the above theorems may be viewed as illustrations of
the well-known principle saying that nuclear spaces often behave
in much the same way as finite-dimensional spaces.

\end{document}